\documentclass[reqno,12pt]{amsart}
\usepackage{amscd}
\usepackage{amsfonts}
\usepackage{amsmath}
\usepackage{amssymb}
\usepackage{amsthm}

\newtheorem{theo}{Theorem}
\newtheorem{theoa}{Theorem}

\newtheorem{def1}{Definition}
\newtheorem{prop}{Proposition}
\newtheorem{lem1}{Lemma}

\def\L{\Lambda}
\def\l{\lambda}

\def\Z{\mathbb{Z}}

\def\R{\mathbb{R}}

\def\1{\mathbf{1}}
\def\eps{\epsilon}

\begin{document}
\title {frame-type families of translates}
\author{shahaf nitzan-hahamov and alexander olevskii}
\date{}

\thanks{Supported in part by the Israel Science Foundation}

\maketitle

\thispagestyle{empty}
\begin{small}
$\!\!\!\!$School of Mathematical Sciences, Tel Aviv University,
Ramat-Aviv, Israel 69978

$\!\!\!\!$e-mail addresses: olevskii@math.tau.ac.il

$\:\:\:\:\:\:\:\:\:\:\:\:\:\:\:\:\:\:\:\:\:\:\:\:\:\:\:\:\:\:\:$
nitzansi@post.tau.ac.il

\end{small}

$\:\:$

\pagestyle{myheadings}

\begin{abstract}
We construct a uniformly discrete, and even sparse, sequence of
real numbers $\L=\{\l_n\}$ and a function $g\in L^2(\R)$, such
that for every $q>2$, every function $f\in L^2(\R)$ can be
approximated with arbitrary small error by a linear combination
$\sum c_n g(t-\l_n)$ with an $l_q$ estimate of the coefficients:
\[
\|\{c_n\}\|_{l_q}\leq C(q)\|f\|.
\]
This can not be done for $q=2$, according to ~\cite{cdh}.

$\:$
\[
\textbf{Keywords:}\textrm{ translates } \cdot \textrm{ frames }
\cdot \textrm{ completeness with estimate of coefficients }
\]
 $\textbf{Mathematics Subject Classification:}$ 42C15 $\cdot$
42C30

\end{abstract}


\section{INTRODUCTION. RESULTS.}
\subsection{}

Let $\Lambda$ be a uniformly discrete set of real numbers:
\begin{equation}
\inf_{\l \neq \l' }|\l-\l'|=\delta>0,\:\:\:\:\:\:\:\:\:\:\:\:
\l,\l' \in
 \L.
\end{equation}

Given a function $g\in L^2(\R)$, consider the family of
translates
\begin{equation}
\{g(t-\l)\}_{\l \in \L}.
\end{equation}
It is well known that for $\L=\Z$ this family cannot be complete
in $L^2(\R)$. It was  conjectured that the same is true for every
uniformly discrete set $\L$ (see, for example, ~\cite{rs}, p.149,
where even a stronger conjecture related to Gabor-type systems
is discussed).
However, this is not the case. The following theorem was proved in
~\cite{ol}:

\begin{theoa}
Let $\L=\{\l_n\}_{n\in \Z}$ be an "almost integer" spectrum:
\begin{equation}
\l_n=n+\alpha_n\,:\:\:\:\:\:\:0<|\alpha_n|\rightarrow
0\:\:\:\:(|n|\rightarrow \infty).
\end{equation}
Then there exists a "generator" $g$ such that the family $(2)$ is
complete in $L^2(\R)$.
\end{theoa}

One may wish to construct a uniformly discrete set of translates
$(2)$ with a stronger property then just completeness. However,
one should keep in mind that no family (2) can be a frame, see
~\cite{cdh}.

\subsection{}

In ~\cite{haol} we introduced an intermediate property between completeness and frame, which is
reproduced here in a slightly different form:

\begin{def1}
We say that a system of vectors $\{u_n\}$ in a Hilbert space $H$
is a $(QF)$-system if the following two conditions are fulfilled:

\medskip
$(i)$ for every $q>2$ there is a constant $C(q)$, such that given
$f\in H$ and $\epsilon>0$, one can find a linear combination
\begin{equation}
Q=\sum c_nu_n,
\end{equation}
satisfying the conditions:
\begin{equation}
\|f-Q\|<\epsilon,
\end{equation}
and
\begin{equation}
\|\{c_n\}\|_{l_q}\leq C(q)\|f\|.
\end{equation}

$(ii)$ $($Bessel inequality$):$
\[
(\sum|\langle f,u_n \rangle|^2)^{\frac{1}{2}}\leq
C'\|f\|,\:\:\:\:\;\:\:\forall f\in H,
\]
where the constant $C'$ does not depend on $f$.
\end{def1}

Approximation property $(i)$ above means "completeness with $l_q$
estimate of coefficients". Using the standard duality argument (see ~\cite{haol}), it can be reformulated   as follows:
\begin{equation}
\|f\|\leq C(p)(\sum|\langle f,u_n
\rangle|^p)^{\frac{1}{p}},\:\:\:\:\:\:\:1/p+1/q=1.
\end{equation}

If the condition $(i)$ in Definition 1 is required for q=2 then it
is identical to usual definition of frames. So, one may regard
$(QF)$-systems as a sort of "quasi-frames".

One may ask about the relation of $ (QF)$ -systems to $p$-frames
defined in ~\cite{ast}. Notice that in the case of Hilbert space,
$p$-frames may only exist for $p=2$, when they are identical to
usual frames.

\subsection{}

In ~\cite{haol} we have constructed sparse exponential systems
\[
E(\L):=\{e^{i\l x}\}_{\l\in \L}
\] which are $(QF)$-systems in
$L^2(S)$ for "generic" sets $S$ of large measure. Observe that these systems cannot be frames, due to
celebrated Landau's  density theorem.

 The goal of this work is to present a similar construction for the
 translates. Our main result is the following

\begin{theo}
There exist a uniformly discrete sequence $ \L=\{\lambda_n \}$
and a function $g\in L^2(\R)$, such that the system $(2)$ is a
$(QF)$-system for $L^2(\R)$.
\end{theo}

We will prove this result in a stronger form, showing that $\L$
can be chosen sparse:

\begin{theo}
Given a sequence of positive numbers $\{\eps_n\}=o(1)$, one can
choose $\L$ in Theorem $1$ so that
\begin{equation}
\frac{\l_{n+1}}{\l_n}>1+\eps_n.
\end{equation}
\end{theo}

Clearly, if $\eps_n$ decreases slowly enough then the gaps in the
spectrum $\L$ grow "almost exponentially". This  condition is sharp,
see  Remark 2.3 below.  Observe that in the context of completeness, this
lacunarity condition  probably first appeared in ~\cite{ol1}, see also
~\cite{olg} and \cite{haol}.

Some remarks on Theorems 1 and 2 are presented in section 2. In
particular, we show (Propositions 1 and 2) that the generator $g$
in these theorems can be chosen infinitely smooth but it cannot
decrease fast at infinity.

\section{PROOFS.}
\subsection{}
Following ~\cite{ol} we start with a reformulation of the result.

For $f\in L^2(\R)$, we denote by $\hat{f}$ its Fourier transform:
$$\hat{f}(t)=\frac{1}{\sqrt{2\pi}}\int_{\R}f(x)e^{ixt}dx.$$

\begin{def1}
A function $w(x) \in L^1(\R)$ is called a
weight if it is strictly positive almost
everywhere on $\R$ (with respect to the Lebesgue measure).
\end{def1}

Let us consider the  weighted space
\[
L^2_w(\R)=\{f: \: \int_{\R}|f(x)|^2w(x)dx <\infty\},
\]
with the scalar product $\langle f, g \rangle =\int
f(x)\bar g(x)w(x)dx$.

Due to the Parseval equality for the Fourier transform, the
transformation
\[
U_w:f\mapsto \widehat{(fw^{\frac{1}{2}})}
\]
is a unitary operator acting from $L^2_w(\R)$ onto $L^2(\R)$.

We set
\[
g(t):=\widehat{(\sqrt{w})}.
\]
Obviously, $g\in L^2(\R)$, and we have:
\[
U_w(e^{i\l x})=g(t-\l).
\]
Since the $(QF)$-property of a system of vectors is invariant with
respect to unitary operators, we  conclude that the system
$\{g(t-\l)\}_{\l \in \L}$ satisfies this property in $L^2(\R)$, if
the system $E(\L):=\{e^{i\l x}\}_{\l \in \L}$ does so in the space
$L^2_w(\R)$. Therefore, Theorem 2 is a consequence of the following
\begin{theo}
There exist a weight $w(x)$ and a uniformly discrete set $\L$
such that
\[
\begin{aligned}
(i)\:\:& \L \textrm{ satisfies }(8)\textrm{ for a pre-given sequence }\{\eps_n\} ;\\
(ii)\:\:& \textrm{The system} \:\:E(\L)\textrm{ is a } (QF)
\textrm{-system }
\textrm{in } L^2_w(\mathbb{R}).\\
\end{aligned}
\]
\end{theo}

$\:$

$\textbf{Remark 2.1.}$ Reversing the argument above, one can deduce
Theorem 3 from Theorem 2, so that these results are equivalent.

$\:$

Our goal now is to prove Theorem 3.

\subsection{}

We need some lemmas.

\begin{lem1}
Let $\L$ be a uniformly discrete sequence and $v$ be a weight such
that $h:=\widehat{(\sqrt{v})}$ is supported by
$[-\frac{\delta}{2}, \frac{\delta}{2}]$ where $\delta$ is the
separation constant defined in $(1)$. Then $E(\L)$ is a Bessel
system in $L^2_v(\R)$.

\end{lem1}

\begin{proof}
The set of translates $\{h(t-\l)\}$ is an orthogonal system of
vectors in $L^2(\R)$ with bounded norms, so it is a Bessel system
in the space. The argument above shows that this system is
obtained from $E(\L)$ by the action of the unitary operator
\[
U_v:L^2_v(\R)\rightarrow L^2(\R).
\]
The lemma follows.
\end{proof}

$\textbf{Remark 2.2.}$ An equivalent definition of a Bessel system
in $H$:
\[
\|\sum a_n u_n\|_H\leq C'\|\{a_n\}\|_{l_2}
\]
(see ~\cite{you}, p.155), clearly extends the result above to every
weight w, such that $w(x)\leq v(x)$ almost everywhere.

$\:$

Given a trigonometric polynomial $Q(x)=\Sigma_{\lambda \in \Lambda}
c_{\lambda}e^{i\lambda x}$, we set
$$\textrm{spec}\,Q=\{\l \in \Lambda:c_\l\neq 0\},$$
and
$$\|Q\|_{q}=\|\{c_{\lambda}\}\|_{l_q}.$$

The following lemma is well known, it goes back to Menshov-type
representation theorems.

\begin{lem1}
Given a segment $I\subset \R$ and a number $\mu >0$, one can find a
trigonometric polynomial $A(x)=\Sigma_{k=1}^K a_ke^{ikx}$ such
that
\begin{align}
(i)\:\:& \|A\|_{{(2+\mu)}}<\mu; \\
(ii)\:\: & m\{x\in I:|A(x)-1|>\mu\}<\mu.
\end{align}

\end{lem1}
For proof see  ~\cite{katz}, Chapter 4 section 2.5, or ~\cite{olg}, Lemma
4.1 and remark 2 on p.382.

\begin{lem1}
Given a segment $I\subset \R$, a number $\xi>0$ and a function
$f\in L^2(I)$, one can find a trigonometric polynomial
$B(x)=\Sigma_{n=1}^N b_ne^{i\beta_nx}$ such that
\begin{align}
(i)\:\: & 0\leq \beta_n-n<\xi\:\:\:\:\:n=1,2,3\ldots N; \\
(ii)\:\:& m\{x\in I:|f(x)-B(x)|>\xi\}<\xi .
\end{align}

\end{lem1}

This can be easily deduced  from Landau's theorem ~\cite{land}, or from ~\cite{ol}.

\begin{lem1}
Let $I\subset \R$ be a segment. For every $\delta>0$ and $f\in
L^2(I)$, there exists a number $l>0$ such that, given an integer
$d>0$, one can find a trigonometric polynomial
$Q(x)=\Sigma_{m=1}^M c_me^{i\lambda_mx}$ which satisfies
\begin{align}
(i)\:\: & \|Q\|_{{(2+\delta)}}<\delta; \\
(ii)\:\: & \lambda_1\geq d;\\
(iii)\:\: & \frac{\lambda_{m+1}}{\lambda_m}>1+l,\:\:\:\:\:\:\:m=1,2,3\ldots M;\\
(iv)\:\: & m \{x\in I:|f(x)-Q(x)|>\delta \}<\delta.
\end{align}

\end{lem1}
\begin{proof}
We can assume that $I=[-\pi s,\pi s]$ for some integer $s>0$.

 Given
$f\in L^2(I)$ and $0<\delta<1$, denote $\xi =\frac{\delta}{2}$ and
use Lemma 3 to find a trigonometric polynomial,
$$B(x)=\Sigma_{n=1}^N b_ne^{i\beta_nx},$$ for which (11) and (12) hold.

Denote
\begin{equation}
\mu=\frac{\delta}{2N\max{\{1,\|B\|_{{(2+\delta)}}\}}},
\end{equation}
and use Lemma 2 to find a trigonometric polynomial,
$$A(x)=\Sigma_{k=1}^K a_ke^{ikx},$$ for which (9) and (10) hold. Set
\begin{equation}
l=\frac{1}{2+K}\:\:.
\end{equation}

Given a positive integer $d>0$, denote
\begin{equation}
 r_n=d(K+1)^{n-1}
 \end{equation}
and define
\[
Q(x)=\Sigma_{m=1}^M c_me^{i\lambda_mx}:=\Sigma_{n=1}^N
b_ne^{i\beta_nx}A(r_nx),
\]
where $\l_1< \l_2< \l_3< \ldots$

Denote
\[
J_n:=\textrm{spec}(e^{i\beta_nx}A(r_nx))
\]
and use (11) to check that $J_{n+1}$ follows $J_n$, for every
$1\leq n< N$. In particular, this fact combined with (9) and (17)
means that
\[
\|Q\|_{{(2+\delta)}}=\|A\|_{{(2+\delta)}}\|B\|_{{(2+\delta)}}\leq\|A\|_{{(2+\mu)}}
\|B\|_{{(2+\delta)}}< \mu \|B\|_{{(2+\delta)}}< \delta.
\]
So property $(i)$ holds for $Q$.

From (19) we have $\l_1=\beta_1+r_1=\beta_1+d,$ so property $(ii)$
follows from (11).

To establish property $(iii)$ note that there are two possible
locations for $\l_m$ and $\l_{m+1}$ in the spectrum of $Q$. First,
they can both belong to $J_n$, for some $1\leq n \leq N$. In this
case
\[
\frac{\l_{m+1}}{\l_m}=\frac{\beta_n+(k+1)r_n}{\beta_n+kr_n}\:\:\:\:\:\:\:\textrm{for
some }1\leq k < K.
\]
On the other hand, $\l_m$ can be the last frequency in $J_n$, for
some $1\leq n \leq N$, while $\l_{m+1}$ is the first frequency in
$J_{n+1}$. In this case
\[
\frac{\l_{m+1}}{\l_m}=\frac{\beta_{n+1}+r_{n+1}}{\beta_n+Kr_n}.
\]
In both cases it is easy to see that (11), (18) and (19) imply
that $(iii)$ holds for $Q$.

To finish the proof we need to show that property $(iv)$ holds for
$Q$. Note that $I=[-\pi s,\pi s]$, so from (10), (12) and (17), we have
\[
m \{x\in I:|f(x)-Q(x)|>\delta\}\leq
\]
\[
m\{x\in I:|f(x)-B(x)|>\frac{\delta}{2}\}+ m \{x\in
I:|B(x)-Q(x)|>\frac{\delta}{2}\}<
\]
\[
\frac{\delta}{2} + m \{x\in I:|\Sigma_{n=1}^N
b_ne^{i\beta_nx}(1-A(r_nx)) |>\frac{\delta}{2}\}\leq
\]
\[
\frac{\delta}{2}+\Sigma_{n=1}^N m \{x\in
I:|b_n(1-A(r_nx))|>\frac{\delta}{2N}\}\leq
\]
\[
\frac{\delta}{2}+N m \{x \in I:
|1-A(x)|>\frac{\delta}{2N\|\{b_n\}\|_{l_{(2+\delta)}}}\}< \delta
\]

which completes the proof.

\end{proof}

We are now ready to prove Theorem 3.

Let $0<\epsilon_n\rightarrow 0$, we can assume that
$\epsilon_{n+1}<\epsilon_n$ for every $n$. Fix an arbitrary weight
$v(x)$ with $\widehat{(\sqrt{v})}$ supported on $[-1, 1]$ and
$\int vdx=1$. Choose a sequence of functions $f_k\in C(\R)$,
$f_k(x)=0$ for every $|x|>k$, which is dense in
$L^2_v(\mathbb{R})$ (as it will automatically be in every space
$L^2_w(\mathbb{R})$, $w\leq v$).

We construct the sequence $\L$ by induction. At the $k$-th step
assume that $\{\l_n\}_{1\leq n < N}$ have already been defined
(where $N=N(k-1)$). For the segment $[-k,k]$, the function $f_k $
and $\delta=\frac{1}{2^k}$, we use Lemma 4 to find a number
$l_k>0$. Let $n_k\geq N$ be the first number for which
\begin{equation}
\eps_{n_k}<l_k.
\end{equation}

Choose arbitrary $\l_N, \l_{N+1}, \l_{N+2},...\l_{(n_k-1)}$ so
that (8) holds for every $n< n_k$ and a number $d_k$ for which
\begin{equation}
\frac{d_k}{\l_{(n_k-1)}}>1+\eps_{(n_k-1)}.
\end{equation}
Use Lemma 4 to find a trigonometric polynomial $Q_k$ so that
properties $(13)-(16)$ hold for $f_k, \delta_k, l_k, d_k$ and
$Q_k$. Add all of $\textrm{spec}\,Q_k$ as a block from the point
$n_k$ forward to form the sequence $\{\l_n\}_{1\leq n < N(k)}$.

Set $\L=\{\l_n\}_{n=1}^{\infty}$ and note that the combination of
properties (14), (15), (20), and (21) ensures that (8) holds for
$\L$.

To define the weight $w$ denote
\begin{equation}
E_k=\{x\in[-k,k]:\:\:|f_k(x)-Q_k(x)|<\delta_k\}.
\end{equation}
From (16) we have
\begin{equation}
m\{[-k,k]\setminus E_k\}<\delta_k.
\end{equation}
Define
\[
 w(x):=v(x) \inf_k
\{1\!\!1_{E_k}(x)+\eta_k1\!\!1_{\mathbb{R}/ E_k}(x)\},
\]
where $1\!\!1_E$ is the indicator function of $E$ and
$\eta_k=(2k\|f_k-Q_k\|_ {L^2_v(\mathbb{R})})^{-2}$. Since
$\Sigma_k\delta_k<\infty$, (23) implies that $0< w \leq v$ almost
everywhere.

Moreover, from (22) we have
\[
\int|f_k-Q_k|^2wdx
<(\delta_k)^2\int_{E_k}vdx+\eta_k\int_{\mathbb{R}/
E_k}|f_k-Q_k|^2vdx,
\]
so
\begin{equation}
\|f_k-Q_k\|_{L^2_w(\mathbb{R})}<\frac{1}{k}.
\end{equation}

Clearly $\L$ is a uniformly discrete sequence. Moreover, we may
suppose $\{\eps_n\}$ to decrease so slowly that the constant in
(1) satisfies $\delta>1$, so we can use Remark 2.2 to deduce that
$E(\L)$ is a Bessel system in $L^2_w(\R)$.

Fix $q>2$. The proof will be complete if we show that, for any
$f\in L^2_w(\mathbb{R})$ with $\|f\|=1$ and $\mu>0$, there exists
a trigonometric polynomial $Q$ such that $\textrm{spec}Q\subset
\L$, $\|f-Q\|_{L^2_w(\mathbb{R})}<\mu$ and $\|Q\|_{q}\leq 1$.

Given such $f$ and $\mu>0$, choose $k$ large enough so that
\begin{equation}
2+\delta_k<q
\end{equation}

\begin{equation}
\|f-f_k\|_{L^2_w\mathbb{R})}<\frac{\mu}{2}
\end{equation}

\begin{equation}
k\mu>2.
\end{equation}

We claim that for the polynomial $Q=Q_k$ all of the properties
described above hold. Indeed, from (13) and (25) we have
\[
\|Q\|_{q}\leq \|Q\|_{{(2+\delta_k)}}\leq \delta_k \leq 1,
\]
while (24) (26) and (27) imply that $\|f-Q_k\|<\mu$. This ends the
proof.$\:\:\:\:\:\:\:\:\:\:\:\:\:\:\:\:\:\:\:\:\:\:\:\:\:\:\:\:\:\:\:\:\:\:\:\:\:\:
\:\:\:\:\:\:\:\:\:\:\:\:\:\:\:\:\:\:\:\:\:\:\:\:\:\:\:\:\:\:\:\:\:\:\:\:\:\:
\:\:\:\:\:\:\:\:\:\:\:\:\:\:\:\:\:\:\:\:\:\:\:\:\:\:\:\:\:\:\:\:\:\:\:\:\:\:\:\:\:\:\:\:\:\:\square$

\subsection{}
Here we discus the "time-frequency" localization of the generator~$g$.
\begin{prop}
As in ~\cite{ol}, one can construct the function $g$ in Theorems 1
and 2 to be infinitely smooth and even the restriction to $\R$ of
an entire function.
\end{prop}
\begin{proof}
Indeed, in the proof of Theorem 3 it is enough to start with a
weight $v_0\leq v$ with sufficiently fast decay, so the same will
be true for $w$. The relation
\[
g(t):=\widehat{(\sqrt(w))}.
\]
implies the required property.
\end{proof}

On the other hand, the weight $w$, constructed in Theorem 3 is
"irregular", which means that the generator $g$ decreases slowly.
This is inevitable, due to the following

\begin{prop}
A generator $g$ in Theorem 1 cannot belong to $L^1(\R)$.
\end{prop}

\begin{proof}
Indeed, suppose it does. Then the corresponding weight
$w(x)=|\hat{g}(x)|^2$ is continues. The system $E(\L)$ is a
$(QF)$-system in $L^2_w(\R)$. Clearly the same property holds in
the space $L^2_w(I)$ for an interval $I$. The set
\[
\{x\in I\:\: : \:\:w(x)>0\}
\]
is an open set of full measure. Take a finite union of intervals
$S\subseteq I$ such that:
\[
mS>\frac{mI}{2}\:\:\: ;\:\:\:\inf_{x\in S}w(x)>0.
\]
Clearly $E(\L)$ is a $(QF)$-system in $L^2(S)$. Now we use Theorem
1 from ~\cite{haol}, where it is proved that if $S$ is a finite
union of intervals and the system $E(\L)$ satisfies the condition
$(i)$ in the Definition 1 (with some $q>2$), then the
Beurling-Landau estimate:
\[
 D^-(\L)\geq \frac{mS}{2\pi}
 \]
 still holds. This contradicts $(1)$, if $mI$ is sufficiently large.

\end{proof}

 As a contrast, notice that in Theorem A, for an
 appropriate uniformly discrete $\L$, the generator $g$ may belong to the Shwartz
 space $S(R)$, see ~\cite{olul2}.

 We conclude by a couple of other remarks.

$\textbf{Remark 2.3.}$ The lacunarity condition in Theorems 2 and
3 is sharp. Indeed, it is well known (see ~\cite{har}) that if
$\L$ is lacunary in the Hadamard sense, that is
$\lambda_{n+1}/\lambda_n
>c>1$, then the system $\{e^{i\l t}\}_{\l \in \L}$  cannot be complete in
$L^2_w(\R)$.

$\textbf{Remark 2.4.}$ By appropriate modification of the proof of
Theorem 3, $\L$ in this result (as well as in Theorems 1 and 2)
can be made  a "small perturbation" of integers, as in the
equality $(3)$. However, again, in contrast to Theorem A, the
perturbations are not arbitrary. In particular $\alpha_n$ cannot
decrease as $O(|n|^{-s})$, $s>0$. This can be proved  similarly to
the corresponding  remark in ~\cite{olg1}.

\end{document}